\documentclass[12pt]{amsart}
\usepackage[T1]{fontenc}
\usepackage{graphicx}
\usepackage{amssymb}
\usepackage{stmaryrd}
\usepackage{enumitem}

\usepackage{geometry}
\usepackage{tikz}
\usepackage{float}
\usepackage{url}
\usepackage{mathtools}
\usepackage{amsmath,amsfonts,amsthm}
\usepackage{mathrsfs} 
\newtheorem{theorem}{Theorem}[section]

\newtheorem{proposition}{Proposition}[section]
\newtheorem{corollary}[theorem]{Corollary}
\newtheorem{definition}{Definition}

\usetikzlibrary{positioning, arrows.meta}

\newcommand{\ew}{\mathrm{ew}}
\newcommand{\SF}{\sigma^*}
\newcommand{\SG}{\sigma(\Gamma)}
\newcommand{\SFG}{\sigma^*(\Gamma)}
\newcommand{\OG}{\overline{\Gamma}}
\author{Duaa Abdullah}
\address{\textbf{Duaa Abdullah:}  Moscow Institute of Physics and Technology (National Research University), 9 Institutskii Lane, 141700 Dolgoprudnyi, Russia}
\email{duaa1992abdullah@gmail.com}
\thanks{}

\title{Degree Variance and the Fuzzy Sigma Index in Fuzzy Graphs}

\date{}
\begin{document}

\begin{abstract}
The sigma index of a graph, defined as the population variance of its degree sequence, is a fundamental measure of structural irregularity. In this paper, we introduce and systematically investigate its natural extension to fuzzy graphs, termed the fuzzy sigma index
$$
\sigma^*(\Gamma) = \frac{1}{n} \sum_{v \in V(\Gamma)} \left( d_\Gamma(v) - \frac{2\,\mathrm{ew}}{n}\right)^2,
$$
where $d_\Gamma(v)$ denotes the fuzzy degree of a vertex $v$, and $\mathrm{ew}$ represents the fuzzy size of the fuzzy graph $\Gamma=(V,\nu, \mu)$. We establish several fundamental properties of this topological index. In particular, we derive sharp lower and upper bounds.
Analyze the behavior of $\sigma^*(\Gamma)$ under standard fuzzy graph operations.  This work provides a foundation for further study of variance-based topological indices in fuzzy graph theory.
\end{abstract}

\maketitle

\noindent\rule{17.5cm}{1.0pt}


\medskip

\textbf{Secondary:} 05C07, 05C09, 05C35, 05C76.

\medskip

\noindent\rule{17.5cm}{1.0pt}

\section{Introduction}\label{sec1}
Throughout this paper, let $\Gamma=(V,E)$ denote a simple, connected graph, where $V(\Gamma)=\{v_1,v_2,\dots,v_n\}$ with $n=|V(\Gamma)|$, and $E(\Gamma)=\{e_1,e_2,\dots,e_m\}$ with $m=|E(\Gamma)|$.  The \textit{composition} of two fuzzy graphs $\Gamma_1[\Gamma_2]$ (also known as the \textit{lexicographic product}) is a construction that combines $\Gamma_1$ and $\Gamma_2$ in a structured manner. Specifically, we create $n_2$ copies of $\Gamma_1$, where $n_2 = |V_2|$. Within each copy, a replica of $\Gamma_2$ is embedded. Connections between vertices in different copies are determined by the structure of $\Gamma_1$.
Topological indices play a significant role in the study of molecular structures~\cite{Gao2016WangMR}. One such index associated with graph irregularity~\cite{Abdo2015Dimitrov} is the \emph{Sigma index}, denoted by $\sigma$. Relationships between various topological indices have been investigated in~\cite{Jahanbai2021Sheikholeslami}, while the Sigma index itself was introduced and further developed in~\cite{gutman2018inverse,Albalahi2023Alanazi,Jahanbanı2019Ediz,Ascioglu2018Cangul}, and is defined as follows:

\begin{equation}~\label{eqqsigmawj00}
\SG=\sum_{uv\in E(G)}(d_G(u)-d_G(v))^2.
\end{equation}

Closely related to Sigma index had introduced by~\cite{Criado2014Flores} as
\begin{equation}~\label{eqqsigmawj01}
\SG=\frac{1}{n}\sum_{i=1}^{n}\left( d_i-\frac{2m}{n}\right)^2.
\end{equation}
The \emph{$\sigma_t$-irregularity} (or sigma total index) of a graph $G$ is defined~\cite{Knor2025RQW,Dimitrov2023Stevanović} as
\[
\sigma_t(G) = \sum_{\{u,v\} \subseteq V(G)} (d_G(u)-d_G(v))^2,
\]
some results on $\sigma_ {t} $-irregularity had given through~\cite{Filipovski2024Dimitrov}. This index has been studied in the context of extremal graph theory, yielding bounds and characterizations of graphs maximizing $\sigma_t$-irregularity.

The primary purpose of this study is to present the fuzzy sigma index, a logical extension of the classical sigma index to fuzzy graphs that is defined as the variance of the fuzzy degree sequence. We intend to investigate its essential aspects, such as sharp bounds, extremal characterisation,  and behaviour under standard fuzzy graph operations.

%%%%%%%%%%%%%%%%%%%%%%%%%%%555
%%%%%%%%%%%%%%%%%%%%%%%%%%%%%5
\subsection{Statement of the Problem}

The \textit{sigma index} of a crisp (ordinary) graph $\Gamma$ with $n$ vertices and $m$ edges is defined as the population variance of its degree sequence. This index quantifies the irregularity of the graph $\Gamma$; it is equal to zero if and only if $\Gamma$ is regular.
In the framework of \textit{fuzzy graph theory}, a natural and significant question arises:

\noindent \textbf{Question 1.} How can the sigma index be appropriately extended to fuzzy graphs?

To answer on the question above, we should be discussed with the following objective:
\begin{enumerate}
\item Define a meaningful fuzzy analogue $\SFG$ of the classical sigma index using the fuzzy degrees
$d_\Gamma(v) = \sum_{u \neq v} \mu(v,u)$ and the fuzzy size $\ew(\Gamma)= \sum_{u < v} \mu(u,v)$.
    
\item Examine the fundamental properties of this index, including:
    \begin{itemize}
        \item Sharp lower and upper bounds,
        \item Extremal graphs attaining the minimum ($\SFG= 0$) and maximum values,
        \item Fuzzy complement graph $\OG$,
        \item Behavior on key classes of fuzzy graphs (such as stars, paths, and regular fuzzy graphs).
    \end{itemize}
\end{enumerate}

Although numerous degree-based topological indices have been extensively investigated in recent years, the variance-based sigma index has not yet been systematically studied in the fuzzy setting.

This work aims to address this gap by introducing a rigorous definition of the fuzzy sigma index and establishing its fundamental extremal and complementary properties.

%%%%%%%%%%%%%%%%%%%%%%%%%%%%%%%%%5

%=====================================
\section{Sigma Index among Fuzzy Graph}
%=====================================
To extend sigma index~\eqref{eqqsigmawj01} rigorously to a \textit{fuzzy graph}, the following definition requires for establishing a fuzzy graph and its fuzzy degrees.
The traditional edge-sum version of the Sigma index in fuzzy graph is discussed~\cite{MordesonMathew2024} such that related literature on fuzzy topological indices.
\begin{definition}[Fuzzy Graph\cite{MordesonMathew2024}]~\label{defFuzzyGraph}
Let $\Gamma=(V, \nu,\mu)$ be a triple graph where $\nu: V \to [0,1]$ is a fuzzy subset of $V$, and $\mu: V \times V \to [0,1]$ is a symmetric fuzzy relation on $\sigma$ satisfying 
    $\mu(u,v) \le \min(\sigma(u),\sigma(v))$ for all $u,v \in V$.
\end{definition}
Thus, the \textit{fuzzy degree} $d(v)$ of a vertex $v \in V(\Gamma)$ and The \textit{fuzzy size} $\mathrm{ew}$ (total fuzzy edge weight) are defined as 
\begin{equation}~\label{eqq1defFuzzyGraph}
d_\Gamma(v):= \sum_{u \in V,\, u \neq v} \mu(v,u), \quad \text{and} \quad \mathrm{ew}:= \sum_{u < v} \mu(u,v).
\end{equation}
Usually, we use $\mathrm{ew}(\Gamma)=\frac{1}{2} \sum_{v \in V} d(v)$.
For establish the \textit{fuzzy sigma index} we will consider the relationship ~\eqref{eqqsigmawj01} by replacement the edges by $\mathrm{ew}$.
\begin{definition}[Fuzzy Sigma Index]~\label{defFuzzysigma}
Let $\Gamma=(V, \nu,\mu)$ be a fuzzy graph. Then, the sigma index of $\Gamma$ is defined as
\begin{equation}~\label{eqq1defFuzzysigma}
\sigma^*(\Gamma) = \frac{1}{n} \sum_{v \in V} \left( d_\Gamma(v)-\lambda \right)^2, 
\end{equation}
where  the average fuzzy degree is $\lambda=2\ew/n$.
\end{definition}
Eq.~\eqref{eqq1defFuzzysigma} has not yet been introduced in the literature. 
The classical relationship of sigma index which treats the \emph{edge-sum} index in fuzzy graph can be obtains based on~\eqref{eqqsigmawj00} as 
\begin{equation}~\label{eqq01defFuzzysigma}
\SF(\Gamma)=\sum_{uv\in E(\Gamma)} \mu(uv)\left(d_\Gamma(u)-d_\Gamma(v)\right)^2.
\end{equation}
According to Definition~\ref{defFuzzysigma}, if $\Gamma$ is crisp graph, then $\sigma(v) \equiv 1$ and $\mu(uv) \in \{0,1\}$ where $uv\in E(\Gamma)$. Thus, from~\eqref{eqq1defFuzzysigma} and according to~\eqref{eqq1defFuzzyGraph} yields  
\begin{equation}~\label{eqq2defFuzzysigma}
\nu^{\SF(\Gamma)}=\frac{1}{\sum_{v \in V(\Gamma)} \nu(v)} \sum_{v \in V(\Gamma)} \nu(v) \left( d_\Gamma(v) -\lambda \right)^2.
\end{equation}
Eq.~\eqref{eqq2defFuzzysigma} establish the concept of \textit{weighted version} of $\Gamma$.  

\medskip 

For example, let $\Gamma$ be a fuzzy graph given by Figure~\ref{fig001fuzzygraph},  consider $V(\Gamma)=\{v_1,v_2,v_3\}$ and the fuzzy adjacency matrix for $\mu$ is
\[
\begin{pmatrix}
0 & 0.8 & 0.3 \\
0.8 & 0 & 0.6 \\
0.3 & 0.6 & 0
\end{pmatrix}.
\]
\begin{figure}[H]
\centering
\begin{tikzpicture}[
    vertex/.style={circle, draw, minimum size=6mm, inner sep=0pt, font=\bfseries},
    edge label/.style={font=\small, fill=red, inner sep=2pt, midway}
]
\node[vertex] (v1) at (0,0) {$v_1$};
\node[vertex] (v2) at (4,0) {$v_2$};
\node[vertex] (v3) at (2,3) {$v_3$};
\draw[line width=4pt, blue!30] (v1) -- (v2) node[edge label, above] {0.8};
\draw[line width=1.1pt, blue!30] (v1) -- (v3) node[edge label, left] {0.3};
\draw[line width=3pt, blue!30] (v2) -- (v3) node[edge label, right] {0.6};
\end{tikzpicture}
\caption{Fuzzy graph with edge thickness is proportional to $\mu(u,v)$.}
\label{fig001fuzzygraph}
\end{figure}
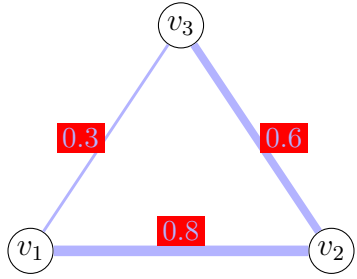
The  fuzzy degrees are $d(v_1) = 0.8 + 0.3 = 1.1$, $d(v_2)= 0.8 + 0.6 = 1.4$ and $d(v_3)= 0.3 + 0.6 = 0.9$. 
The fuzzy size is $\ew(\Gamma)= 0.8 + 0.3 + 0.6 = 1.7$. Thus, for $n=3$ we obtain $\lambda=17/15$. Then, $\sigma(\Gamma)=19/450$. 

%%%%%%%%%%%%%%%%%%%%%%%%%%%%%%%%%%%%%%%%%%%%
\subsection{Behavior under Fuzzy Operations}
%%%%%%%%%%%%%%%%%%%%%%%%%%%%%%%%%%%%%%%%%%%%
Let $\Gamma_1=(V_1, \nu_1, \mu_1)$ and $\Gamma_2=(V_2, \nu_2, \mu_2)$ be two fuzzy graphs with $|V_1|=n_1$ and $|V_2|=n_2$, fuzzy sizes $\ew_1(\Gamma)$ and $\ew_2(\Gamma)$, and average fuzzy degrees $\lambda_1$ and $\lambda_2$. 
Consider a combined graph $\Gamma=\Gamma_1*\Gamma_2$. We denote the number of vertices of $\Gamma$ by $n$, where $n=n_1+n_2$ or $n=n_1\,n_2$, depending on the graph operation.

\begin{proposition}[Union $\Gamma_1\cup \Gamma_2$]~\label{Unionfuzzy01}
Let $\Gamma_1=(V_1, \nu_1, \mu_1)$ and $\Gamma_2=(V_2, \nu_2, \mu_2)$ be two fuzzy graphs. Then
\begin{equation}~\label{eqq1Unionfuzzy01}
\SF(\Gamma_1\cup \Gamma_2)=\frac{n_1\, \SF(\Gamma_1)+n_2\, \SF(\Gamma_2)}{n}+ \frac{n_1\,n_2}{n^2}(\lambda_1-\lambda_2)^2.
\end{equation}
\end{proposition}
According to Theorem~\ref{fuzzyExtremal01}, if $\Gamma_1=(V_1, \nu_1, \mu_1)$ and $\Gamma_2=(V_2, \nu_2, \mu_2)$ is two  fuzzy  regular graphs. Then, $\SF(\Gamma_1\cup \Gamma_2)=0$.

\begin{proposition}[Join $\Gamma_1\vee \Gamma_2$]~\label{Unionfuzzy02}
Let $\Gamma_1=(V_1, \nu_1, \mu_1)$ and $\Gamma_2=(V_2, \nu_2, \mu_2)$ be two fuzzy graphs. The degrees are 
\begin{equation}~\label{eqq1Unionfuzzy02}
\begin{aligned}
&d_{\Gamma_1\vee \Gamma_2}(u)=d_{\Gamma_1}(u)+\sum_{v\in V(\Gamma_2), \, u\in V(\Gamma_1)}\min\{\nu_{\Gamma_1}(u),\nu_{\Gamma_2}(v) \}\\
&d_{\Gamma_1\vee \Gamma_2}(v)=d_{\Gamma_2}(v)+\sum_{u\in V(\Gamma_1), \,v\in V(\Gamma_2) }\min\{\nu_{\Gamma_1}(u),\nu_{\Gamma_2}(v) \}.
\end{aligned}
\end{equation}
\end{proposition}

\begin{proposition}[Cartesian Product $\Gamma_1\square \Gamma_2$]~\label{Unionfuzzy03}
Let $\Gamma_1=(V_1, \nu_1, \mu_1)$ and $\Gamma_2=(V_2, \nu_2, \mu_2)$ be two fuzzy graphs. Then
\begin{equation}~\label{eqq1Unionfuzzy03}
\begin{aligned}
&d_{\Gamma_1\square \Gamma_2}(u,v)=d_{\Gamma_1}(u)+d_{\Gamma_2}(v)\\
&\ew(\Gamma)=n_2\,\ew_1(\Gamma)+n_1\,\ew_2(\Gamma).
\end{aligned}
\end{equation}
\end{proposition}
Actually, according to Proposition~\ref{Unionfuzzy03}, we obtain $d_{\Gamma_1\times \Gamma_2}(u,v)=d_{\Gamma_1}(u)\,.\,d_{\Gamma_2}(v)$. 
Let $\Gamma_1 = (V_1, \nu_1, \mu_1)$ and $\Gamma_2 = (V_2, \nu_2, \mu_2)$.
The fuzzy composition $\Gamma= \Gamma_1[\Gamma_2]$ is defined with 
the vertex set is $V = V_1 \times V_2$ (see Figure~\ref{fig001composition}), consisting of all ordered pairs $(u, v)$ with $u \in V_1$ and $v \in V_2$. For any two vertices $(u_1, v_1)$ and $(u_2, v_2)$: If $u_1 = u_2$ (same layer) $\mu_G\bigl((u_1, v_1), (u_2, v_2)\bigr) = \mu_2(v_1, v_2)$ and if  $u_1 \ne u_2$ (different layers) $ \mu_G\bigl((u_1, v_1), (u_2, v_2)\bigr) = \mu_1(u_1, u_2)$. In the second case, the edge membership depends only on the first component and is independent of $v_1$ and $v_2$. 
Then, the degrees $d(u,v)$ of $\Gamma_1[\Gamma_2]$ are $d(u,v)=n_2\, d_{\Gamma_1}(u)+d_{\Gamma_1}(v)$. 

\begin{figure}[H]
    \centering
    \begin{tikzpicture}[
    vertex/.style={circle, draw, thick, minimum size=3mm, inner sep=1pt, font=\small},
    layer/.style={rectangle, draw, dashed, rounded corners, minimum width=3cm, minimum height=3.2cm, align=center},
    arrow/.style={->, thick, -Latex, gray},
    label/.style={font=\footnotesize, fill=white, inner sep=1pt}
]

\node[vertex, fill=blue!20] (u1) at (-1,2) {$u_1$};
\node[vertex, fill=blue!20] (u2) at (1,2) {$u_2$};

\draw[line width=1pt, blue!60] (u1) -- (u2) node[label, above] { };

\node[vertex, fill=green!20] (v1) at (5.5,2.2) {$v_1$};
\node[vertex, fill=green!20] (v2) at (7,1)   {$v_2$};
\node[vertex, fill=green!20] (v3) at (8.5,2.2){$v_3$};

\draw[line width=1pt, green!70] (v1) -- (v2) node[label, below] { };
\draw[line width=1pt, green!70] (v2) -- (v3) node[label, below] { };
\draw[line width=1pt, green!70] (v3) -- (v1) node[label, left] { };

\node[layer, fill=blue!5] (L1) at (0,-1) {};
\node[vertex] (a1) at (-1,-0.5) {$u_1v_1$};
\node[vertex] (a2) at (0,-1.8)  {$u_1v_2$};
\node[vertex] (a3) at (1,-0.5)  {$u_1v_3$};

\draw[line width=1pt, green!60] (a1) -- (a2) node[label, below] { };
\draw[line width=1pt, green!60] (a2) -- (a3) node[label, below] { };
\draw[line width=1pt, green!60] (a3) -- (a1) node[label, above] { };

\node[layer, fill=blue!5] (L2) at (7,-1) {};
\node[vertex] (b1) at (6,-0.5)  {$u_2v_1$};
\node[vertex] (b2) at (7,-1.8)  {$u_2v_2$};
\node[vertex] (b3) at (8,-0.5)  {$u_2v_3$};

\draw[line width=2pt, green!60] (b1) -- (b2) node[label, below] { };
\draw[line width=2pt, green!60] (b2) -- (b3) node[label, below] { };
\draw[line width=2pt, green!60] (b3) -- (b1) node[label, above] { };

% Strong connections between layers (according to G1)
\draw[line width=1pt, blue!70] (a1) -- (b1) node[label, above] { };
\draw[line width=1pt, blue!70] (a1) -- (b2) node[label, right] { };
\draw[line width=1pt, blue!70] (a1) -- (b3) node[label, below] { };

\draw[line width=1pt, blue!70] (a2) -- (b1);
\draw[line width=1pt, blue!70] (a2) -- (b2);
\draw[line width=1pt, blue!70] (a2) -- (b3);

\draw[line width=1pt, blue!70] (a3) -- (b1);
\draw[line width=1pt, blue!70] (a3) -- (b2);
\draw[line width=1pt, blue!70] (a3) -- (b3);
\end{tikzpicture}
    \caption{Fuzzy composition $\Gamma= \Gamma_1[\Gamma_2]$ graph.}
    \label{fig001composition}
\end{figure}
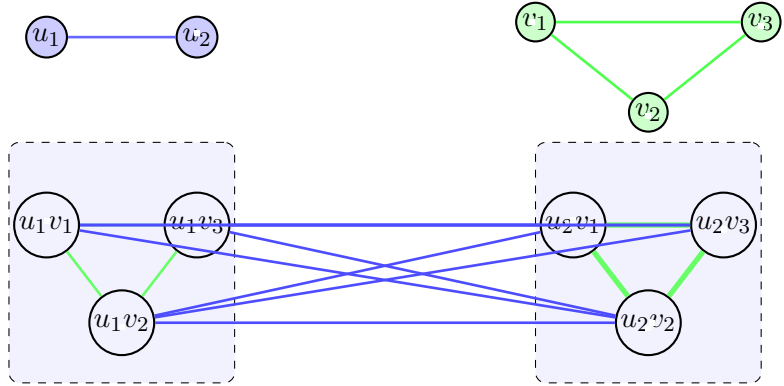
As a result, the irregularity of $\Gamma_1[\Gamma_2]$ is primarily determined by the irregularity of the outer graph $\Gamma_1$. When $n_2$ is high, $\SF(\Gamma_1[\Gamma_2])\geqslant \SF(\Gamma_1)$. 

\begin{theorem}~\label{fuzzyCartesian001}
For the Cartesian product $\Gamma = \Gamma_1 \square \Gamma_2$, the fuzzy sigma index is additive
\begin{equation}~\label{eqq1fuzzyCartesian001}
\SF(\Gamma) = \sigma(\Gamma_1) + \sigma(\Gamma_2).
\end{equation}
\end{theorem}
\begin{proof}
For any vertex $(u,v)$ in $\Gamma$, its degree satisfies $d_{\Gamma}(u,v) = d_{\Gamma_1}(u)+d_{\Gamma_2}(v)$. Since the degree is formulated as the sum of two independent sequences, according to Proposition~\ref{Unionfuzzy03}, which establishes~\eqref{eqq1fuzzyCartesian001}. 
\end{proof}

%================================================
\subsection{Upper Bounds on Fuzzy Sigma Index}
%================================================
We begin with simple bounds on $\SF$ by Proposition~\ref{Boundsfuzzy01} based on Maximum degree $\Delta$ and the minimum degree $\delta$. 
\begin{proposition}~\label{Boundsfuzzy01}
Let $\Gamma=(V, \nu,\mu)$ be a fuzzy graph on $n$ vertices with  maximum degree $\Delta(\Gamma)$ and minimum degree $\delta(\Gamma)$. Then
\begin{equation}~\label{eqq1Boundsfuzzy01}
\SF(\Gamma)\leqslant 2n+\frac{(\Delta-\delta)^2}{4}.
\end{equation}
\end{proposition}
\begin{proof}
For any fuzzy degree $\delta\leqslant d_\Gamma(v)\leqslant\Delta$ and $\SF(\Gamma)\geq 0$  where $v\in V(\Gamma)$. Then, we should be prove the following claim. 

\noindent \textbf{Claim 1.} Let $a_1,a_2,\dots,a_k$ be real number satisfying 
\begin{equation}~\label{eqq2Boundsfuzzy01}
\frac{1}{k}\sum_{i=1}^{k}(a_i-\bar{a})^2\leqslant \frac{(x-y)^2}{4},    
\end{equation}
where $x,y\in \mathbb{R}$. 

Assume $a_k\in [x,y]$ for all $k\in \mathbb{R}$ and let $\bar{a}$ be the average of $a_1,a_2,\dots,a_k$. Consider the function $f(t)=(t-\bar{a})^2$, note that 
\begin{equation}~\label{eqq3Boundsfuzzy01}
\begin{aligned}
(a_k-\bar{a})^2&=(a_k-x)^2-(\bar{a}-x)^2\\
&\leqslant \max\{(a_k-x)^2,(\bar{a}-x)^2 \}\\
&\leqslant (y-x)^2.
\end{aligned}
\end{equation}
Thus, from~\eqref{eqq3Boundsfuzzy01} had provided a trivial bounds. Now, let $y_k=a_k-x$ where $0\leq y_i\leq \ell$ and $\ell=y-x$. Thus, 
\begin{equation}~\label{eqq4Boundsfuzzy01}
0\leq \frac{1}{k}\sum_{i=1}^{k}(y_i-\bar{y})^2\leq \ell. 
\end{equation}
Then, 
\begin{align*}
\sum_{i=1}^{k}(y_i-\bar{y})^2 &=\sum_{i=1}^{k}y_i^2-n\,\bar{y}^2\\
&\leqslant \ell \sum_{i=1}^{k}y_i \\
&=n\,\bar{y}\ell.
\end{align*}
Since $\lvert a_i-a_j\rvert \leqslant \Delta-\delta$ where $0\leq i\leq j\leq k$ we obtain $(a_i-a_j)^2\leqslant (\Delta-\delta)^2$. Thus, according to~\eqref{eqq4Boundsfuzzy01}, 
\begin{equation}~\label{eqq5Boundsfuzzy01}
\SF(\Gamma)=\frac{1}{2k}\sum_{0\leq i\leq j\leq k}(d_i-d_j)^2 \leqslant \frac{(\Delta-\delta)^2}{2}.
\end{equation}
Thus, the relationship~\eqref{eqq2Boundsfuzzy01} holds. 
Now, assume $v_r$ be a vertices has degree $\Delta$ where $r\in \mathbb{R}$ and $k-v_r$ vertices has degree $\delta$. Then, 
\begin{equation}~\label{eqq6Boundsfuzzy01}
\bar{a}=\frac{v_r\,\Delta+(k-v_r)\,\delta}{k},
\end{equation}
from~\eqref{eqq5Boundsfuzzy01} and \eqref{eqq6Boundsfuzzy01} we obtain 
\begin{equation}~\label{eqq7Boundsfuzzy01}
\begin{aligned}
\SF(\Gamma)&=\frac{v_r}{k}\left(\Delta-\bar{a}\right)^2+\frac{k-v_r}{k}\left(\delta-\bar{a}\right)^2\\
&=\frac{v_r(k-v_r)}{k^2}(\Delta-\delta)^2.
\end{aligned}
\end{equation}
Thus, from~\eqref{eqq7Boundsfuzzy01} by considering the quadratic function $f(v_r)=v_r(k-v_r)$ is maximum if $v_r=k/2$. Then,  the relationship~\eqref{eqq1Boundsfuzzy01} holds.
\end{proof}

Let $\Gamma=\mathcal{P}_n$ be a fuzzy path with vertices ordered $v_1-v_2-\cdots-v_n$. In this case, for $i=1,\dots,n-1$, and $\mu(v_i,v_j)=0$, otherwise only consecutive edges have positive membership values: 
$\mu(v_i, v_{i+1}) > 0$.  The degrees are $d(v_1) = \mu(v_1,v_2)$, \quad $d(v_n) = \mu(v_{n-1},v_n)$. The  degrees of internal vertices are  $d(v_i) = \mu(v_{i-1},v_i) + \mu(v_i,v_{i+1})$ for $i=2,\dots,n-1$.
In next theorem, we established the upper bound of $\SF$. Actually, the equality~\eqref{eqq1conjecturefuzzy} holds for fuzzy star graphs.

\begin{theorem}~\label{conjecturefuzzy}
Let $\Gamma=(V, \nu,\mu)$ be a fuzzy graph. Then,
\begin{equation}~\label{eqq1conjecturefuzzy}
\SF(\Gamma) \leqslant \dfrac{(n-1)^2 \Big(2\ew(\Gamma)\Big)^2}{n^3}.
\end{equation}
\end{theorem}
\begin{proof}
Assume the maximum possible value of any single degree in any fuzzy graph is $\Delta(\Gamma)\leqslant  2\ew(\Gamma)$. Since the degree sequence is as unbalanced, when one vertex carries the full degree size  $2\ew(\Gamma)$ and the remaining vertices $n-1$ share the complementary degrees such that their mutual connections are zero. Thus, $\sum_{v\in V(\Gamma)}d_\Gamma(v)$ is maximised. Now, let 
\begin{equation}~\label{eqq2conjecturefuzzy}
\sum_{v\in V} d(v)^2 = \sum_{v\in V(\Gamma)}\left(\sum_{u\ne v}\mu(v,u)\right)^2.
\end{equation}
When there is a dominating vertex, the cross terms are maximised by expanding and applying non-negativity of $\mu$. Assume $d_1\geqslant d_2\geqslant \dots \geqslant d_n$ be a fuzzy degrees of $\SF$. Since  $d_1\leqslant 2\,\ew(\Gamma)$
and $\sum_{i=2}^n d_i=2\,\ew(\Gamma)-d_1$, we consider 
\begin{equation}~\label{eqq3conjecturefuzzy}
\sum d_i^2 = d_1^2 + \sum_{i=2}^n d_i^2 + 2(d_1 - d_i)^2,
\end{equation}
because for fixed sum $\lambda_1=2\,\ew(\Gamma)-d_1$, from~\eqref{eqq3conjecturefuzzy} the sum of squares $\sum_{i=2}^n d_i^2$ is
maximized when one $d_i$ among $\lambda_1$ and others are $0$. Therefore, the relationship~\eqref{eqq1conjecturefuzzy} holds in fuzzy graphs this is achievable only if a fuzzy star graph.

Thus, we emphasize that maximum of $\sum_{v\in V(\Gamma)}d_\Gamma(v)^2$ satisfying $\sum_{v\in V(\Gamma)}d_\Gamma(v)=\lambda_1+d_1$ and the leaves have degrees equals to $\lambda_1+d_1$. 
Assume the center vertex $c$ and the leaf set $L=V(\Gamma)-\{c\}$. Then, 
\[
d_\Gamma(c)=\sum_{\ell\in L}\mu(c,\ell)=\lambda_1+d_1.
\]
In this case,
\begin{equation}~\label{eqq4conjecturefuzzy}
\sum_{v\in V(\Gamma)}d_\Gamma(v)^2=\left(\lambda_1+d_1\right)^2+\sum_{\ell \in L} d_\Gamma(\ell)^2.
\end{equation}
Since each leaf degree $d_\Gamma(\ell)\leqslant 2\,\ew(\Gamma)$, we get
\begin{equation}~\label{eqq5conjecturefuzzy}
\SF(\Gamma)= \frac{(2\,\ew(\Gamma))^2 + \sum_{\ell \in L} d_\Gamma(\ell)^2}{n} - \left(\frac{2\,\ew(\Gamma)}{n}\right)^2.
\end{equation}
Thus, from~\eqref{eqq2conjecturefuzzy}--\eqref{eqq5conjecturefuzzy} the relationship~\eqref{eqq1conjecturefuzzy} holds.
\end{proof}

Let $\Gamma=\mathcal{S}_n$ be a fuzzy star with center vertex $c$ and leaf set $L$ where $|L|=n-1$. Then, according to Theorem~\ref{conjecturefuzzy}  if $\mu(u,v) = 0$ where no edges exist among the leaves for all $u,v \in L$.
The sum of squares of degrees satisfies the identity:
\[
\sum_{v \in V} d_\Gamma(v)^2=\Big(2\,\ew(\Gamma)\Big)^2+\sum_{\ell \in L} d_\Gamma(\ell)^2.
\]

Assume uniform edge membership $\alpha > 0$ on the existing edges. Then $\ew(\Gamma)=(n-1)\alpha$,  $2\,\ew(\Gamma)=2\,(n-1)\alpha$, and  $\lambda=2(n-1)\alpha/n$.
In this case, $d_\Gamma(c)=2\,(n-1)\alpha$, and $d(\ell) = \alpha$ for each leaf. Then, 
\begin{equation}~\label{eqq6conjecturefuzzy}
 \SF(\mathcal{S}_n)=\frac{1}{n}\Bigl[\bigl(2(n-1)\alpha-\lambda\bigr)^2+(n-1)(\alpha-\lambda)^2\Bigr].
\end{equation}

Hence, according to~\eqref{eqq6conjecturefuzzy} consider $\ew(\Gamma)=(n-1)\alpha$, $d(v_1) = d(v_n) = \alpha$, and $d(v_i) = 2\alpha$ for $i = 2, \dots, n-1$. Thus,  
\begin{equation}~\label{eqq7conjecturefuzzy}
\SF(\mathcal{P}_n)\frac{1}{n} \Bigl[ 2(\alpha - \lambda)^2 + (n-2)(2\alpha - \lambda)^2 \Bigr].
\end{equation}

Next,  by considering \eqref{eqq6conjecturefuzzy} and \eqref{eqq7conjecturefuzzy} for establishing the upper bound of $\SF$ we discussed fuzzy degree in non-increasing order. 
\begin{theorem}~\label{theoremfuzzzy001}
Let $\Gamma=(V, \nu,\mu)$ be a fuzzy graph on $n$ vertices with fuzzy degree in non-increasing order $d_1\geqslant d_2\geqslant \dots \geqslant d_n$.   Then, 
\begin{equation}~\label{eqq1theoremfuzzzy001}
\SF(\Gamma)\leqslant \frac{(n-1)(d_1-\lambda)}{n}.
\end{equation}
\end{theorem}
\begin{proof}
Since $d_1\leqslant \sum_{j=2}^{n}\mu(v_1v_j)\leqslant (n-1)\max{\mu}$ and $d_1\leqslant \ew(\Gamma)$ we obtain 
\begin{equation}~\label{eqq2theoremfuzzzy001}
\SF(\Gamma)\leqslant \lambda^2\left(1-\frac{1}{n}\right).
\end{equation}
Thus, if $\Delta\leqslant 2\, \ew(\Gamma)$, then for a vertex $v\in V(\Gamma)$ satisfying $d_\Gamma(v)\leqslant \sum_{u\neq v} \min\{\mu(u),\mu(v)  \}.$ Then $d_\Gamma(v)^2=d_\Gamma(v)\,.\,d_\Gamma(v_1)$ if and only if $d_\Gamma(v)\in \{ 0, d_\Gamma(v_1)\}$. Thus, 
\begin{align*}
\SF(\Gamma) &\leqslant \sum_{v\in V(\Gamma)}d_\Gamma(v)^2\\
&\leqslant d_\Gamma(v_1) \sum_{v\in V(\Gamma)}d_\Gamma(v)\\
&\leqslant d_\Gamma(v_1)\,n\,\lambda.
\end{align*}
Hence, dividing by $n$, we obtain 
\begin{equation}~\label{eqq3theoremfuzzzy001}
\SF(\Gamma) \leqslant  \lambda(d_1-\lambda).
\end{equation}
Thus, from~\eqref{eqq2theoremfuzzzy001} and \eqref{eqq3theoremfuzzzy001} the relationship~\eqref{eqq1theoremfuzzzy001} holds.
\end{proof}

Keep in mind that such fuzzy networks do exist; a fuzzy star, in which the remaining vertices are only connected to the center and one central vertex bears all the edge membership weight, is the canonical example.

%================================================
\subsection{Lower Bounds on Fuzzy Sigma Index}
%================================================

The fuzzy irregularity idea can be employed to produce a positive lower bound for non-regular fuzzy graphs.

\begin{proposition}~\label{lowerBoundsfuzzy01}
Let $\Gamma=(V, \nu,\mu)$ be a non-regular fuzzy graph on $n$ vertices with  maximum degree $\Delta(\Gamma)$ and minimum degree $\delta(\Gamma)$. Then
\begin{equation}~\label{eqq1lowerBoundsfuzzy01}
\SF(\Gamma)\geqslant  \frac{1}{n^2}\left(\Delta-\delta\right)^2.
\end{equation}
\end{proposition}
\begin{proof}
Since $\Gamma$ is irregular fuzzy graph. Let $d_1,\dots,d_n$ denote the fuzzy degrees. Then, according to Proposition~\ref{Boundsfuzzy01}, 
\begin{equation}~\label{eqq2lowerBoundsfuzzy01}
\SF(\Gamma)=\frac{1}{2n^2} \sum_{i=1}^n \sum_{j=1}^n (d_i - d_j)^2.
\end{equation}
Then, from~\eqref{eqq2lowerBoundsfuzzy01} there exist indices $i_0$ and $j_0$ such that $d_{i_0} = \Delta$ and $d_{j_0} = \delta$. Thus, it satisfies with $(d_{i_0} - d_{j_0})^2 = (\Delta - \delta)^2,$ and $(d_{j_0} - d_{i_0})^2 = (\Delta - \delta)^2$. Hence,
\begin{equation}~\label{eqq3lowerBoundsfuzzy01}
\sum_{i=1}^n \sum_{j=1}^n (d_i - d_j)^2 \geqslant 2(\Delta - \delta)^2.
\end{equation}
Eq.~\eqref{eqq3lowerBoundsfuzzy01} follows that the relationship~\eqref{eqq1lowerBoundsfuzzy01} hold.
\end{proof}

%============================================
\subsection{Fuzzy Complement Graph}
%=============================================
Let $\overline{\Gamma}$ be the fuzzy complement of $\Gamma=(V,\nu,\mu)$ where $\overline{\mu}(uv)=\min(\nu(u),\nu(v))-\mu(uv)$ (see Figure~\ref{fig002fuzzygraph}). Theorem~\ref{fuzzyComplement01} had established the relationship between $\Gamma$ and $\overline{\Gamma}$

\begin{theorem}~\label{fuzzyComplement01}
Let $\overline{\Gamma}$ be the fuzzy complement of $\Gamma=(V,\nu,\mu)$. If $\nu(v)=1$ for every $v\in V(\Gamma)$, then $\SF(\OG)=\SFG$. Then, 
\begin{equation}~\label{eqq1fuzzyComplement01}
 \SFG+\SF(\OG)=2\,\SFG, \quad     \SFG\,.\,\SF(\OG)=(\SFG)^2.
\end{equation}
\end{theorem}
\begin{proof}
Since $\overline{\mu}(uv)=\min(\nu(u),\nu(v))-\mu(uv)$. Then, if $\sigma\equiv1$ it follows that $d_{\OG}(v)=(n-1)-d_\Gamma(v)$. Thus, $\lambda_{\OG}=(n-1)-\lambda$. Hence, $d_{\OG}(v)-\lambda_{\OG}=-(d_G(v)-\lambda).$ Consider $$d_{\max}(v)=\sum_{u\neq v}\min(\nu(v),\nu(u)).$$ 
Then, by considering that $d_{\OG}(v)\,.\,d_{\max}(v)-d_\Gamma(v)$, $\delta(v)=d_\Gamma(v)-\lambda$ and $e(v)=d_{\max}(v)-\lambda_{\max}$. Then, 
\begin{equation}~\label{eqq2fuzzyComplement01}
\begin{aligned}
\SF(\OG)&=\frac{1}{n}\sum_{v\in V(\OG)}(e(v)-\delta(v))^2\\
&=\frac{1}{n}\sum_{v\in V(\OG)}e(v)^2-\frac{2}{n}\sum_{v\in V(\OG)}e(v)\delta(v)+\frac{1}{n}\sum_{v\in V(\OG)}\delta(v)^2\\
&=\frac{1}{n}\sum_{v\in V(\OG)}e(v)^2-\frac{2}{n}\sum_{v\in V(\OG)}e(v)\delta(v)+\SFG.
\end{aligned}
\end{equation}
Thus, from~\eqref{eqq2fuzzyComplement01} the relationship~\eqref{eqq1fuzzyComplement01} holds.  
\end{proof}
Furthermore, by Corollary~\ref{fuzzycorollary01} we emphasize that the upper bound. It demonstrates that when the underlying structure is very imbalanced in both directions (star-like), the total degree variance of a fuzzy graph and its complement is precisely maximised. The irregularity of fuzzy graphs and their complements can be measured using this natural benchmark.

\begin{corollary}~\label{fuzzycorollary01}
According to Theorem~\ref{conjecturefuzzy}, 
\begin{equation}~\label{eqq1fuzzycorollary01}
0\leq \SFG+\SF(\OG)\leqslant \frac{8(\ew(\Gamma))^2(n-1)}{n^2}, 
\end{equation}
with equality on the right-hand side holds if and only if $\Gamma\cong \mathcal{S}_n$ and $\OG\cong \mathcal{S}_n$.
\end{corollary}

\begin{corollary}~\label{fuzzycorollary02}
According to Theorem~\ref{fuzzyComplement01}, the upper and lower bounds satisfying 
\begin{itemize}
\item Lower bound, 
\begin{equation}~\label{eqq1fuzzycorollary02}
\SFG+\SF(\OG) \geqslant 2\,\SFG+\frac{1}{n}\sum_{v\in V(\OG)}e(v)^2- 2\sqrt{\SFG\,.\,\frac{1}{n}\sum_{v\in V(\OG)}e(v)^2}.    
\end{equation}
\item Upper bound
\begin{equation}~\label{eqq2fuzzycorollary02}
\SFG+\SF(\OG) \leqslant 2\,\SFG+\frac{1}{n}\sum_{v\in V(\OG)}e(v)^2+ 2\sqrt{\SFG\,.\,\frac{1}{n}\sum_{v\in V(\OG)}e(v)^2}.   
\end{equation}
\end{itemize}
\end{corollary}

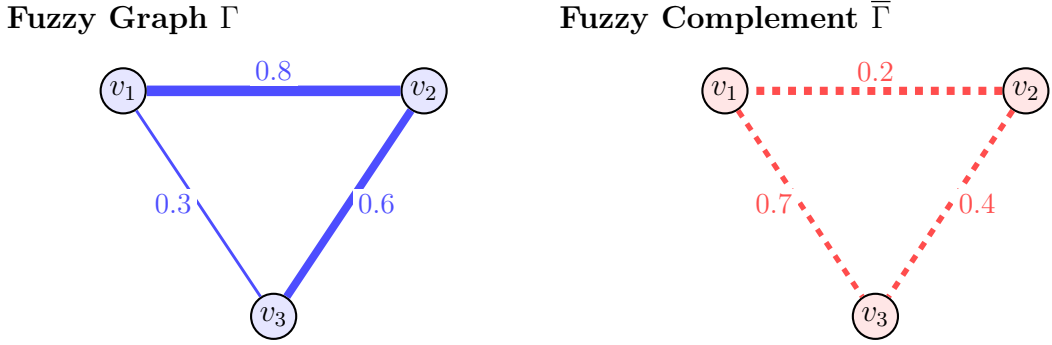
\begin{figure}[H]
    \centering
\begin{tikzpicture}[
    vertex/.style={circle, draw, thick, minimum size=6mm, inner sep=0pt, font=\bfseries},
    edge/.style={thick},
    label/.style={font=\small, fill=white, inner sep=2pt, midway}
]
% ========== Original Fuzzy Graph ================
\node[vertex, fill=blue!10] (A1) at (0,3) {$v_1$};
\node[vertex, fill=blue!10] (A2) at (4,3) {$v_2$};
\node[vertex, fill=blue!10] (A3) at (2,0) {$v_3$};

\draw[edge, line width=4pt, blue!70] (A1) -- (A2) node[label, above] {0.8};
\draw[edge, line width=1.1pt, blue!70] (A1) -- (A3) node[label, left] {0.3};
\draw[edge, line width=3pt, blue!70] (A2) -- (A3) node[label, right] {0.6};

\node[above=0.3cm of A1, font=\bfseries] {Fuzzy Graph $\Gamma$};

% =========== Fuzzy Complement Graph ==============
\node[vertex, fill=red!10] (B1) at (8,3) {$v_1$};
\node[vertex, fill=red!10] (B2) at (12,3) {$v_2$};
\node[vertex, fill=red!10] (B3) at (10,0) {$v_3$};

\draw[edge, line width=2pt, red!70, dashed] (B1) -- (B3) node[label, left] {0.7};
\draw[edge, line width=3.1pt, red!70, dashed] (B2) -- (B1) node[label, above] {0.2};
\draw[edge, line width=2.1pt, red!70, dashed] (B2) -- (B3) node[label, right] {0.4};

\node[above=0.3cm of B1, font=\bfseries] {Fuzzy Complement $\OG$};

%\draw[->, thick, gray, line width=1.1pt] (5.5,1.5) -- (6.5,1.5) 
%    node[midway, above, font=\small] {Complement: $\overline{\mu}(uv) = 1 - \mu(uv)$};
\end{tikzpicture}
\caption{Fuzzy complement graph.}
    \label{fig002fuzzygraph}
\end{figure}

%=======================================================
\section{Extremal Characterization of Fuzzy Sigma Index}
%=======================================================
We have now totally solved the second outstanding problem: characterise the fuzzy graphs that obtain the minimum and maximum values of $\SFG$. Theorem~\ref{fuzzyExtremal01} had provided  the fuzzy-regular graph.

\begin{theorem}~\label{fuzzyExtremal01}
Let $\Gamma=(V, \nu,\mu)$ be a fuzzy graph. Then, $\SFG=0$ if and only if $G$ is fuzzy-regular, that is, there exists a constant $r \ge 0$ such that $d_\Gamma(v)=r$ for every $v \in V(\Gamma)$. 
\end{theorem}
\begin{proof}
By definition, $\SFG$ is the population variance of the sequence $\{d_\Gamma(v)\}_{v \in V(\Gamma)}$. The variance is zero if and only if all elements of the sequence are equal. In this case, $\lambda=d_\Gamma(v)=r$. 
\end{proof}

A natural concern in the study of fuzzy graphs is how irregular a graph can be given a particular overall fuzzy size. The irregularity measure $\SFG$ reflects the dispersion of vertex degrees, thus it's important to discover its greatest possible value and characterise the extremal structures that achieve it. Theorem~\ref{fuzzyExtremal02} offers a sharp upper bound on $\SFG$ and a detailed description of the extremal fuzzy graphs.

\begin{theorem}~\label{fuzzyExtremal02}
Let $\Gamma=(V, \nu,\mu)$ be a fuzzy graph on $n \geqslant 2$ vertices. Then
\begin{equation}~\label{eqq1fuzzyExtremal02}
\SFG\leqslant \frac{2\,\ew(\Gamma)^2}{n} \left(1-\frac{2}{n}\right),
\end{equation}
with equality if and only if $\Gamma$ contains exactly one positive fuzzy edge  with membership value $\ew(\Gamma)$, and all other memberships are zero.
Then, there exist distinct vertices $u, v$ such that $\mu(u, v) = \ew(\Gamma)$ and $ \mu(x, y) = 0$ for all other pairs $\{x, y\} \neq \{u, v\}$ where  the degree sequence is $\{\ew(\Gamma), \ew(\Gamma), 0, \dots, 0\}$, consisting of $n-2$ zeros.
\end{theorem}
\begin{proof}
According to Theorem~\ref{conjecturefuzzy}, we find that 
\[
\SF(\Gamma) \leqslant \dfrac{(n-1)^2 \Big(2\ew(\Gamma)\Big)^2}{n^3}.
\]
Let $d_1, \dots, d_n \geq 0$ be real numbers with fixed sum $\alpha=2\,\ew(\Gamma)$. Since 
\begin{equation}~\label{eqq2fuzzyExtremal02}
\sum_{i=1}^n d_i^2 \leq 2\,\ew(\Gamma)^2, 
\end{equation}
The upper bound~\eqref{eqq2fuzzyExtremal02} is attained if and only if the degree sequence is $\{\ew(\Gamma), \ew(\Gamma), 0, \dots, 0\}$. This occurs precisely when all membership is concentrated on a single pair of vertices $u, v$ with $\mu(u, v)= \ew(\Gamma)$. In this case, we emphasize that $d_\Gamma(u)=\ew(\Gamma)$, $d_\Gamma(v)=\ew(\Gamma)$ and $d_\Gamma(w)=0$ where $w \notin \{u, v\}.$ Thus, 
\begin{align*}
\SFG &=\frac{2(\ew(\Gamma)-\lambda)^2+(n-2)(0-\lambda)^2}{n}\\
&=\frac{2\ew(\Gamma)^2}{n}\left(1-\frac{2}{n}\right).
\end{align*}
Therefore, according to~\eqref{eqq2fuzzyExtremal02} the relationship~\eqref{eqq1fuzzyExtremal02} holds.
\end{proof}

\begin{corollary}~\label{fuzzycorollary03}
For fixed $\ew(\Gamma)$, the most irregular fuzzy graphs are precisely the fuzzy single-edge graphs, consisting of one fuzzy edge of weight $\ew(\Gamma)$ and $n-2$ isolated vertices. 
\end{corollary}
In contrast, based on Corollary~\ref{fuzzycorollary03} fuzzy star graphs with multiple leaves are not extremal and yield strictly smaller values of $\SFG$.

\medskip 

Let $\Gamma$ be a fuzzy graph on $n=6$ vertices  such that is \textit{fuzzy 4-regular}.  We build it as the union of two edge-disjoint fuzzy 2-regular graphs (i.e., two fuzzy cycles) with proper memberships.
Let the vertices be $v_1,v_2,v_3,v_4,v_5,v_6$. Suppose we assign constant edge membership $\alpha = 0.4$ on the edges of two disjoint cycles such as Figure~\ref{fig001regfuzzy}. Then each vertex has exactly four incident fuzzy edges:
\[
d_\Gamma(v_i)=0.4+0.4+0.2=1.0 \quad \text{for all } i=1,\dots,6.
\]
The fuzzy size is $\ew(\Gamma)=3\times 0.4 + 3\times 0.4 + 3\times 0.2=3$ and $\lambda=1$. Since all fuzzy degrees are equal to $\lambda$, we have $\SFG=0$. 

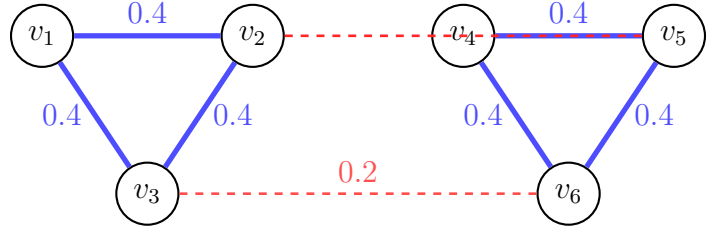
\begin{figure}[H]
\centering
\begin{tikzpicture}[scale=1.4]
    % Vertices
    \node[circle, draw, thick, minimum size=1mm] (v1) at (0,2) {$v_1$};
    \node[circle, draw, thick, minimum size=1mm] (v2) at (2,2) {$v_2$};
    \node[circle, draw, thick, minimum size=1mm] (v3) at (1,0.5) {$v_3$};
    
    \node[circle, draw, thick, minimum size=1mm] (v4) at (4,2) {$v_4$};
    \node[circle, draw, thick, minimum size=1mm] (v5) at (6,2) {$v_5$};
    \node[circle, draw, thick, minimum size=1mm] (v6) at (5,0.5) {$v_6$};

    % Cycle 1 (thick blue)
    \draw[line width=2pt, blue!70] (v1) -- (v2) node[midway, above] {0.4};
    \draw[line width=2pt, blue!70] (v2) -- (v3) node[midway, right] {0.4};
    \draw[line width=2pt, blue!70] (v3) -- (v1) node[midway, left] {0.4};

    % Cycle 2 (thick blue)
    \draw[line width=2pt, blue!70] (v4) -- (v5) node[midway, above] {0.4};
    \draw[line width=2pt, blue!70] (v5) -- (v6) node[midway, right] {0.4};
    \draw[line width=2pt, blue!70] (v6) -- (v4) node[midway, left] {0.4};

%    \draw[line width=1pt, red!70, dashed] (v1) -- (v4) node[midway, above] {0.2};
\draw[line width=1pt, red!85, dashed] (v2) -- (v5) node[midway, right] { };
    \draw[line width=1pt, red!70, dashed] (v3) -- (v6) node[midway, above] {0.2};
\end{tikzpicture}
\caption{A fuzzy 4-regular graph on 6 vertices with $\SFG=0$.}~\label{fig001regfuzzy}
\end{figure}

\section{Conclusion}

This paper introduced and systematically analyzed the fuzzy sigma index, defined as the population variance of the fuzzy degree sequence of a fuzzy graph $\Gamma$.  We established several fundamental properties of this index. In particular, we proved that $\SF(\Gamma) = 0$ if and only if $\Gamma$ is fuzzy-regular. We further showed that the maximum value of $\SFG$ is attained precisely when $\Gamma$ consists of a single fuzzy edge carrying the entire membership weight. Additionally, a sharp upper bound was derived in terms of the fuzzy size $\ew$ and the number of vertices.

Investigated the behavior of the fuzzy sigma index under several key fuzzy graph operations. Notably, the Cartesian product satisfies the additive property $\SF(\Gamma_1 \square \Gamma_2) = \SF(\Gamma_1) + \SF(\Gamma_2)$. The effects of union, join, tensor product, and composition were likewise examined, revealing distinct patterns in how irregularity propagates through these operations.
Consequently, this work establishes a foundation for the study of variance-based topological indices in fuzzy graph theory. Future research directions include weighted extensions of the fuzzy sigma index, its connections with other fuzzy topological indices, and potential applications in fuzzy network analysis.


\begin{thebibliography}{99}
%==========================
\bibitem{Abdo2015Dimitrov} H. Abdo, D. Dimitrov, Non-regular graphs with minimal total irregularity, {\em Bull. Aust. Math. Soc.} \textbf{92} (2015) 1--10.
\bibitem{Albalahi2023Alanazi} A. Ali, A. M. Albalahi, A. M. Alanazi, A. A. Bhatti,
A. E. Hamza, On the maximum sigma index of k-cyclic graphs,
{\em Discrete Appl. Math.} \textbf{325} (2023) 58--62,
\url{https://doi.org/10.1016/j.dam.2022.10.009}.

\bibitem{Ascioglu2018Cangul} M. Ascioglu, I. N. Cangul, Sigma index and forgotten
index of the subdivision and r-subdivision graphs,
{\em Proc. Jangjeon Math. Soc.} \textbf{21}(2) (2018) 1--14.

\bibitem{Criado2014Flores} R. Criado, J. Flores, A. G. del Amo, M. Romance,
Centralities of a network and its line graph: an analytical comparison by
means of their irregularity, {\em Int. J. Comput. Math.} \textbf{91} (2014)
304--314, \url{https://doi.org/10.1080/00207160.2013.793316}.

\bibitem{Jahanbai2021Sheikholeslami} Z. Du, A. Jahanbai, S. M. Sheikholeslami,
Relationships between Randic index and other topological indices,
{\em Commun. Comb. Optim.} \textbf{6}(1) (2021) 137--154.

\bibitem{Dimitrov2023Stevanović} D. Dimitrov, D. Stevanović, On the
$\sigma_t$-irregularity and the inverse irregularity problem,
{\em Appl. Math. Comput.} \textbf{441} (2023) \#127709.

\bibitem{Filipovski2024Dimitrov} S. Filipovski, D. Dimitrov, M. Knor, R. Škrekovski,
Some results on $\sigma_t$-irregularity, {\em arXiv e-prints} (2024),
arXiv:2411.

\bibitem{Gao2016WangMR} W. Gao, W. Wang, M. R. Farahani, Topological indices study
of molecular structure in anticancer drugs, {\em J. Chem.} (2016),
\url{http://dx.doi.org/10.1155/2016/3216327}.

\bibitem{Gutman2018I} I. Gutman, Topological indices and irregularity measures,
{\em Bull. Int. Math. Virtual Inst.} \textbf{8} (2018) 469--475.

\bibitem{gutman2018inverse} I. Gutman, M. Togan, A. Yurttas, A. S. Cevik, I. N. Cangul,
Inverse problem for sigma index, {\em MATCH Commun. Math. Comput. Chem.}
\textbf{79}(3) (2018) 491--508.

\bibitem{Jahanbanı2019Ediz} A. Jahanbanı, S. Ediz, The sigma index of graph operations,
{\em Sigma J. Eng. Nat. Sci.} \textbf{37}(1) (2019) 155--162.

\bibitem{Knor2025RQW} M. Knor, R. Škrekovski, S. Filipovski, D. Dimitrov,
Extremizing antiregular graphs by modifying total $\sigma$-irregularity,
{\em Appl. Math. Comput.} \textbf{490} (2025) \#129199,
\url{https://doi.org/10.1016/j.amc.2024.129199}.

\bibitem{MordesonMathew2024} J. N. Mordeson, S. Mathew, Fuzzy Mathematics, Graphs,
and Similarity Measures, {\em Elsevier}, Chapter 12 – Sigma index (2024)
161--182.

\end{thebibliography}
\end{document}